\documentclass[a4paper,11pt]{article}
\usepackage[utf8]{inputenc}
\usepackage{lslmath}
\usepackage[numbers]{natbib}
\usepackage{mathtools}
\usepackage{csquotes}
\newcommand{\Gnm}{G_n^M}

\usepackage{tikz}
\usetikzlibrary{automata, arrows, positioning, calc}
\definecolor{keynoteblue}{HTML}{00A2FF}

\newcommand{\ran}{\operatorname{Ran}}

\newcommand{\tran}{\text{Ran}}
\newcommand{\I}{\mathbb{I}}
\renewcommand{\ind}{\mathbb{I}}

\begin{document}

\title{Clique and cycle frequencies in a sparse random graph model with overlapping communities}
\author{Tommi Gröhn \and Joona Karjalainen \and Lasse Leskelä}
\date{Aalto University, Espoo, Finland\\ 
  \vspace{\baselineskip} \today}
\maketitle

\begin{abstract}
A statistical network model with overlapping communities can be generated as a superposition of mutually independent random graphs of varying size. The model is parameterized by the number of nodes, the number of communities, and the joint distribution of the community size and the edge probability. This model admits sparse parameter regimes with power-law limiting degree distributions and non-vanishing clustering coefficients. This article presents large-scale approximations of clique and cycle frequencies for graph samples generated by the model, which are valid for regimes with unbounded numbers of overlapping communities. Our results reveal the growth rates of these subgraph frequencies and show that their theoretical densities can be reliably estimated from data.
\end{abstract}

\section{Introduction}

Subgraph frequencies are commonly used as descriptive statistics of graphs, and are of particular interest in testing and fitting statistical network models \cite{ Frank_1979, ouadah2020degree, Picard_Daudin_Koskas_Schbath_Robin_2008}. It seems intuitive that the similarity between two observed graphs, or an observed graph and a random graph, could be measured to some extent with the numbers of small structures found in them. 
Moreover, the emergence of structures such as cliques has natural interpretations in social networks and gives insight into how networks are organized  \cite{Benson_Gleich_Leskovec_2016, Tsourakakis_Pachocki_Mitzenmacher_2017,wegner2018identifying}. Subgraph frequencies have also been used for parameter estimation by employing the method of moments, as in  \cite{Ambroise_Matias_2012, Bickel_Chen_Levina_2011, karjalainen2017moment, karjalainen2018parameter}, where one derives the expected values of suitable empirical quantities and solves the resulting nonlinear equations.

Deriving the expected subgraph densities in random graphs is generally not a trivial task, 
except for special cases, such as the Erd\H{o}s--R{\'e}nyi model $G(n,p)$. The expressions can become complicated and laborious to work with even for simple models, as illustrated in \cite{Allman_Matias_Rhodes_2011} for a class of random graph mixture models. 
Such expressions can, however, provide insight into the structure of the random graph and its parameters, and it would therefore be beneficial to have them readily available.

In this paper we study cliques and cycles, two natural classes of connected subgraphs. Cliques can be seen as basic building blocks of community structures in networks, and their role in different types of networks has been widely studied. Since cycles are subgraphs with two paths between each pair of nodes, they represent redundant information in the structure of the network. Specialized counting algorithms exist for both types of subgraphs.

We study an undirected random graph model with overlapping communities. The model is motivated by the structure of social networks, although the concept of communities is naturally extended to other types of networks \cite{guillaume2004bipartite}. The model can be viewed as a superposition of independent Erd\H{o}s--R{\'e}nyi random graphs in a sparse regime. When the number of communities $m$ grows linearly with $n$, the model introduces non-trivial dependencies between the edges in the graph: in the limit $n \to \infty$, the model allows for non-trivial clustering coefficients and degree distributions \cite{Bloznelis_Leskela_2019}, as well as tractable bidegree distributions and assortativity coefficients \cite{Bloznelis_Karjalainen_Leskela_2022}.  Clearly, the asymptotic structure of this model is very different from $G(n,p)$.
On the other hand, when $m$ grows faster than $n$, the clustering coefficient can approach zero \cite{Bloznelis_Leskela_2019}. We mainly focus on the setting where $m$ is of the same order as $n$, which is known to yield the richest model structure \cite{Bloznelis_Leskela_2019,Bloznelis_Karjalainen_Leskela_2022,vadon2019new}.

In this paper, we derive expressions for the asymptotic expected frequencies of cliques and cycles, and prove the concentration of the corresponding empirical densities under moment conditions. This reveals the growth rates of these frequencies, and shows that the theoretical subgraph densities can be reliably estimated from data. The proofs are mainly based on combinatorial and graph-theoretical arguments and employ the second-moment method. The generality of the results is in contrast with some earlier work, where each subgraph has been analyzed separately as needed. 

Random intersection graphs \cite{Bergman_Leskela_2022,Bloznelis_Kurauskas_2017,Dong_Hu_2023,Godehardt_Jaworski_2001} are obtained as a special case of our model when all communities (ER graphs) have edge probability one.  They form an important class of statistical models capable of generating sparse graph samples with power-law degree distributions and non-trivial clustering (i.e., triangle counts of the same order as edge counts), which are commonly encountered when analyzing real-world networks \cite{Watts_Strogatz_1998}.  Random intersection graphs with $m_n \asymp n$ admit such a rich statistical structure \cite{Bloznelis_2013,Bloznelis_Leskela_2019}, whereas random intersection graphs with $m_n \gg n^3$ become indistinguishable from ER graphs \cite{Brennan_Bresler_Nagaraj_2020}. However, random intersection graphs generate large numbers of cliques of various sizes \cite{Rybarczyk_Stark_2010}, and this may induce a poor fit to real-world data.  To address this issue, cliques in random intersection graphs can be replaced by ER graphs, which leads to the model studied in the present article.
Similar models have been studied in varying levels of generality in \cite{Bloznelis_Leskela_2016, Bloznelis_Leskela_2019, Bloznelis_Karjalainen_Leskela_2022, Petti_Vempala_2018-02, yang2012community}. In particular, our model can be seen as a stochastic and parsimonious variation of the AGM model studied in \cite{yang2012community}. In contrast with some earlier work, such as \cite{Bloznelis_Karjalainen_Leskela_2022}, we do not assume that the edge probabilities converge to a constant. We also mention that a closely related model, also with non-trivial edge probabilities within the communities, has been proposed in \cite{vadon2019new}.

\section{Model description}
We define an undirected random graph model on $n$ nodes with $m$ (possibly) overlapping layers. 
Fix integers $n, m \geq 1$, and a probability distribution $\pi_n$ on $\{0,\dots,n\}\times [0,1]$.
The layers $G_{n,1},\dots,G_{n,m}$ are mutually independent random graphs. For each layer $G_{n,k}$, $k=1,\ldots,m$, the size $X_{n,k}$ and strength $Y_{n,k}$ are drawn from the distribution $\pi_n$, and the node set $V(G_{n,k})$ is then drawn uniformly at random from the subsets of $\{1,\ldots,n\}$ of size $X_{n,k}$. The edges in layer $k$ are generated independently with probability $Y_{n,k}$, so that for a graph $g$ with $V(g) \subset \{1, \ldots, n\}$ we have
\[
\pr (G_{n,k} = g) = \int_{\{|V(g)|\} \times [0,1]} \binom{n}{x}^{-1} y^{|E(g)|}(1-y)^{\binom{x}{2}-|E(g)|} \, \pi_n(dx,dy).
\]

Given the list of layers $G_{n,1},\dots, G_{n,m}$, the random graph $G_n$ is defined as an undirected graph with node set $V(G_n) = \{1,\dots,n\}$ and edge set $E(G_n) = \cup_{k=1}^m E(G_{n,k})$. 

It should be noted that the strength of the layer is allowed to depend on its size.  The independent case with $X_{n,k} \sim \Bin(n,p)$ and $Y_{n,k}=1$ (for some $p \in (0,1)$) reduces to the widely studied random intersection graph. When $Y_{n,k}$ is a deterministic constant in $(0,1)$ and $X_{n,k} \sim \Bin(n,p)$, we obtain the thinned random intersection graph studied in \cite{karjalainen2018parameter}.

\section{Definitions and notation}
\label{sec:Assumptions}
We define the subgraph frequency (or \new{count}) as follows.
Let $G$ be a graph on node set $[n] := \{1,\dots,n\}$, and let $K_n$ be the complete graph on $[n]$. Denote by $\Sub(R,G)$ the collection of $R$-isomorphic subgraphs of $G$.
The subgraph frequency of $R$ in $G$ is then
\[
 N_R(G) \weq \sum_{R' \in \Sub(R,K_n)} \nquad \ind(G \supset R').
\]
In particular, not only induced subgraphs are counted according to this definition; for example, the triangle graph $K_3$ contains three copies of the 2-path $K_{1,2}$.
Throughout this paper we assume that $R$ does not depend on $n$.

We model a large network using a sequence of random graphs $(G_n: n \ge 1)$ indexed by the graph size $n$, so that $G_n$ is parameterized by $(n, m_n, \pi_n)$. 
For a distribution $\pi$ and a random vector $(X,Y) \sim \pi$ we write 
\[
(\pi)_{r,s} = \E[(X)_r Y^s],
\]
where $(X)_r = X(X-1)\ldots(X-r+1)$. 
For a graph $R=(V, \{e_1, \ldots, e_k\})$, we denote the number of edges by $|E| = k$, and the number of nodes incident to the edges  by $\norm{E}$, i.e.,
\[
\norm{E} \weq | \{v: \, \exists i \in [k] \, \text{ s.t. } \, v \in e_i  \} |.
\]
For convenience, the scale parameter $n$ is often omitted and we denote, e.g., $m=m_n$ and $G = G_n$. 

We follow the asymptotic notations of \cite{Janson_Luczak_Rucinski_2000}: For nonnegative sequences $a_n, b_n$, we write
$a_n = o(b_n)$ when $\limsup_{n \to \infty} a_n/b_n = 0$,
$a_n = O(b_n)$ when $\limsup_{n \to \infty} a_n/b_n < \infty$,
and $a_n = \Theta(b_n)$ when $a_n = O(b_n)$ and $b_n = O(a_n)$.
To avoid extensive use of parentheses, we also employ the abbreviations
$a_n \ll b_n$ for $a_n = o(b_n)$, $a_n \lesim b_n$ for $a_n = O(b_n)$,
and $a_n \asymp b_n$ for $a_n = \Theta(b_n)$.
Finally, we use the standard notation $X_n = o_\pr (1)$ to denote that $X_n \to 0$ in probability.

\section{Results}

\subsection{General subgraph frequencies}

Our first result concerns the expected frequency of a subgraph with $r$ nodes and $s$ edges.
As $n\to \infty$, the expected frequency depends heavily on the growth rate of the cross-moment $(\pi_n)_{r,s}$. The intuition behind this result is that small subgraphs are commonly formed by a single layer, and the expected number of such subgraphs is of order $m_n(\pi_n)_{r,s}$. The contribution of subgraphs formed by multiple layers is at most of the same order (but not necessarily negligible). This is guaranteed by the condition in \eqref{eq:maxcond} below, which states that the subgraphs of $R$ should not be too common within the layers, as compared to $R$ itself.

\begin{theorem}
\label{thm:mainexpectedvalue}
Let $R$ be a connected graph with $r$ nodes and $s$ edges, and assume that
the layer type distribution satisfies
\begin{equation}
 \label{eq:maxcond}
 \max_{\emptyset \ne E \subset E(R)} \left( \frac{m_n}{n} (\pi_n)_{\norm{E}, \abs{E}} \right)^{1/\abs{E}}
 \wlesim \left( \frac{m_n}{n} (\pi_n)_{r,s} \right)^{1/s}.
\end{equation}
Then the expected number of $R$-isomorphic subgraphs in the model satisfies
$\E N_R(G_n) \asymp m_n (\pi_n)_{r,s} \wedge n^r$.
\end{theorem}

In the following two remarks we discuss special cases where the condition in \eqref{eq:maxcond} can be verified for an arbitrary connected graph $R$. The first case with nonrandom layer strengths is a generalization of the setting in (\cite{karjalainen2018parameter},   Assumption 1). In the latter case, $\pi_n \to \pi$, the conditions can be further simplified for specific types of $R$, as we will show in Section \ref{sec:cliqcyc}.

\begin{remark*}
Let the layer strengths $Y_{n,k}$ be nonrandom and equal to some $y_n \in (0,1)$.
Then $(\pi_n)_{a,b} = (\pi_n)_a y_n^b$ where $(\pi_n)_a = (\pi_n)_{a,0}$.
When $m \asymp n$, Condition~\eqref{eq:maxcond} is equivalent to
$\max_{\emptyset \ne E \subset E(R)} (\pi_n)_{\norm{E}}^{1/\abs{E}} \lesim (\pi_n)_r^{1/s}$.
A simple sufficient condition for this is to assume that $(\pi_n)_r \asymp 1$.
\end{remark*}

\begin{remark*}
Assume that  $\pi_n \to \pi$ weakly with $(\pi_n)_{\norm{E}, \abs{E}} \to (\pi)_{\norm{E}, \abs{E}} < \infty$ for all $E \subset E(R)$, and some limiting layer type distribution $\pi$ such that $(\pi)_{r,s} \in (0,\infty)$.
When $m_n \lesim n$, \eqref{eq:maxcond} holds and we conclude that $\E N_R(G_n) \asymp m_n$.
\end{remark*}

Our second result verifies that the empirical subgraph frequencies are close to their expected values in a probabilistic sense when $m_n \asymp n$. In particular, it shows that the relative errors $N_R(G_n)/\E N_R(G_n)-1$ of the subgraph density estimates converge to zero in probability. Results of this type are used in the proofs of convergence for moment-based parameter estimators in, e.g., \cite{Ambroise_Matias_2012, karjalainen2018parameter}, and we will discuss an example of this type of estimator at the end of Section \ref{sec:cliqcyc}.
The assumption
$(\pi_n)_{r,s} \gg n^{-1}$
below guarantees that the expected value tends to infinity with $n$. If
$(\pi_n)_{r,s} \ll n^{-1}$,
the result would be trivial, as  $\E N_R(G_n) \to 0$ by Theorem \ref{thm:mainexpectedvalue}.
In the following we denote 
\begin{equation}
 \label{eq:pinmdef}
 (\pi_n^M)_{a,b}
 = \E \Big((X_n)_a Y_n^b \I(X_n>M) \Big)
 \qquad \text{where $(X_n,Y_n) \sim \pi_n$},
\end{equation}
and note that the condition in \eqref{eq:UI} below is similar to the definition of uniform integrability.

\begin{theorem}
\label{thm:mainvariance}
Let $R$ be a connected graph with $r$ nodes and $s$ edges.
Assume that $m_n \asymp n$ and that the layer type distribution satisfies
$(\pi_n)_{r,s} \gg n^{-1}$ and $(\pi_n)_{0,2s} \lesim (\pi_n)_{r,s}^2$ together with
\begin{equation}
 \label{eq:maxcondLinear}
 \max_{\emptyset \ne E \subset E(R)} (\pi_n)_{\norm{E}, \abs{E}}^{1/\abs{E}}
 \wlesim (\pi_n)_{r,s}^{1/s}
\end{equation}
and
\begin{equation}
 \label{eq:UI}
 \sup_n \max_{E \subset E(R)} (\pi_n^M)_{\norm{E}, \abs{E}}
 \wto 0
 \qquad \text{as $M \to \infty$}.
\end{equation}
Then the number of $R$-isomorphic subgraphs in the model is concentrated according to $N_R(G_n) = \left(1 + o_{\pr}(1) \right) \E N_R(G_n)$.
\end{theorem}

\subsection{Clique and cycle frequencies}
\label{sec:cliqcyc}
We now turn to the special case where $R$ is a clique or a cycle. The estimation of these subgraph densities was the original motivation for this paper, as they form natural classes of subgraphs to use for moment-based estimators. We apply Theorem \ref{thm:mainvariance} to prove the concentration of empirical clique and cycle frequencies under more easily verifiable conditions. In particular, we avoid the conditions above that involve all subsets $E \subset E(R)$ explicitly.

In Theorem \ref{the:piconverges} below we assume that the layer type distribution $\pi_n$ converges weakly to a limit distribution with sufficiently many moments.  
In Theorem \ref{the:sizeconverges} we give another sufficient set of conditions, namely that the layer strengths are deterministic and that the moments of $X_n$ are sufficiently well-behaved. At the end of this section, we discuss some implications of these results.

\begin{assumption} \label{ass:piconverges} If $r\geq2$ is even, assume that as $n \to \infty$,
\[
(\pi_n)_{r,r/2} \to (\pi)_{r,r/2} \in (0, \infty).
\]
If $r$ is odd, assume that 
\begin{align*}
(\pi_n)_{r-1,\frac{r-1}{2}} \to (\pi)_{r-1,\frac{r-1}{2}} &\in (0, \infty), \quad \text{and} \\ (\pi_n)_{r,\frac{r+1}{2}} \to (\pi)_{r,\frac{r+1}{2}} &\in (0, \infty).
\end{align*}
\end{assumption}

\begin{remark*}
The conditions in Assumption \ref{ass:piconverges}   correspond to 
\[
(\pi_n)_{\norm{E}, \abs{E}} \to (\pi)_{\norm{E}, \abs{E}} \in (0, \infty),
\] where $E$ is a set of $r/2$ or $(r-1)/2$ disjoint edges, or a union of a disjoint set of edges and a 2-path. These conditions together with {$m_n \asymp n$} control the numbers of $r$-cliques and $r$-cycles formed by multiple layers.
\end{remark*}

\begin{theorem}
\label{the:piconverges}
Assume that $m_n \asymp n$, $\pi_n$ converges weakly to a distribution $\pi$ on $\mathbb{Z}_+ \times [0,1]$, and that Assumption \ref{ass:piconverges} holds for some $r \ge 2$.
\begin{enumerate}[(i)]
\item If $(\pi)_{r,\binom{r}{2}} >0$, then $(\pi_n)_{r,\binom{r}{2}} \to (\pi)_{r,\binom{r}{2}}$,
and the $r$-clique frequency in the model satisfies
\[
\E(N_{K_r}(G_n)) = (1+O(n^{-1})) (r!)^{-1} m_n (\pi_n)_{r,\binom{r}{2}} ,
\]
and
\[
N_{K_r}(G_n) = (1+o_\pr (1)) \E(N_{K_r}(G_n)).
\]
\item If $(\pi)_{r,r} >0$, then $(\pi_n)_{r,r} \to (\pi)_{r,r}$,
and the $r$-cycle frequency in the model satisfies
\[
\E(N_{C_r}(G_n)) = (1+O(n^{-1})) (2r)^{-1} m_n (\pi_n)_{r,r} ,
\]
and
\[
N_{C_r}(G_n) = (1+o_\pr (1)) \E(N_{C_r}(G_n)).
\]
\end{enumerate}
\end{theorem}

\begin{remark*}
A simple example of a sequence $(\pi_n)$ that satisfies the assumptions of Theorem \ref{the:piconverges} (for arbitrary $r$) is given by
\[
X_n \sim \Bin(n,5/n), \quad Y_n = (1+X_n)^{-1}, \quad (X_n, Y_n) \sim \pi_n,
\]
which results in many small and dense layers, as well as larger and sparser ones.
\end{remark*}

Next, we focus on a case where the layer strengths are deterministic. Here the conditions of Theorem \ref{thm:mainvariance} can be quite easily verified without considering each subset $E$ in \eqref{eq:UI} separately. We note that the set of assumptions we use in Theorem \ref{the:sizeconverges} is  completely different from Theorem \ref{the:piconverges}: weak convergence of $\pi_n$ is no longer assumed, and the moments (of $X_n$) are assumed to be bounded, but not to necessarily converge. Thus, the model may in fact depend heavily on $n$ through both $X_n$ and $Y_n$ without losing the concentration properties of the clique and cycle frequencies.

\begin{assumption} \label{ass:product} Let $(X_n, Y_n) \sim \pi_n$. For given integers $r$ and $s$, assume that the collection $\{ (X_n)_r: \, n = 1,2,\ldots \}$ is uniformly integrable with $\E[(X_n)_r] = \Theta(1)$, and $Y_n = p_n \in [0,1]$ is such that $p_n^s \gg n^{-1}$. 
\end{assumption}

\begin{theorem}
\label{the:sizeconverges}
Assume that $m_n \asymp n$.
\begin{enumerate}[(i)]
\item If Assumption \ref{ass:product} holds for some $r \geq 2$ and $s= \binom{r}{2}$, then
\[
\E(N_{K_r}(G_n)) = (1+O(n^{-1})) (r!)^{-1} m_n \E[(X_n)_r] p_n^{\binom{r}{2}} ,
\]
and
\[
N_{K_r}(G_n) = (1+o_\pr (1)) \E(N_{K_r}(G_n)).
\]
\item If Assumption \ref{ass:product} holds for some $r\geq 3$ and $s=r$, then
\[
\E(N_{C_r}(G_n)) = (1+O(n^{-1})) (2r)^{-1} m_n \E[(X_n)_r] p_n^r ,
\]
and
\[
N_{C_r}(G_n) = (1+o_\pr (1)) \E(N_{C_r}(G_n)).
\]
\end{enumerate}
\end{theorem}

We now discuss some implications of these results. Consider a sequence of layer size distributions $X_1, X_2, \ldots$ satisfying Assumption \ref{ass:product} for some $r$. If $Y_n = p_n \ll n^{-1/\binom{r}{2}}$, then by Theorem \ref{thm:mainexpectedvalue} and Markov's inequality the number of $r$-cliques is zero with high probability, and the same is true for $r$-cycles if $p_n \ll n^{-1/r}$. This suggests that if $p_n \to 0$ with a specific rate, then $r$-cycles appear in the model, but $r$-cliques do not. Indeed, it follows from Theorem \ref{the:sizeconverges} that this is the case for 4-cycles and 4-cliques when $n^{-1/4} \ll p_n \ll n^{-1/6}$. This phenomenon is not observed in the passive random intersection graph (i.e., $p_n=1$): namely, if $p_n$ is a constant, it follows from Theorem \ref{the:sizeconverges} that both $N_{K_4}$ and $N_{C_4}$ are of order $n$.

On the other hand, consider an Erd\H{o}s--R{\'e}nyi graph $G(n,p_{ER})$ with the same edge probability as our model, i.e., 
\[
p_{ER} = 1-\left(1-\frac{\E[(X)_2]p_n}{(n)_2}\right)^m.
\]
When $m \asymp n$ and $\E[(X)_2] \asymp 1$, it follows that, $p_{ER} \asymp p_n/n$, and so $\E N_{C_4} \asymp p_n^4$, and $\E N_{K_4} \asymp n^{-2}p_n^6$ in the ER graph. In our model, under the conditions of Theorem \ref{thm:mainexpectedvalue}, $\E N_{C_4} \asymp np_n^4$ and $\E N_{K_4} \asymp np_n^6$. This example illustrates the fact that our model typically generates much more cliques and cycles than $G(n,p_{ER})$.

Finally, we consider a parameter estimation problem. Let the  strength of the layers be a constant, $Y_n = p \in (0,1)$. Then, under the assumptions of Theorem \ref{the:sizeconverges}, $\E N_{K_4}/\E N_{C_4} = p^2/3(1+O(n^{-1}))$, which suggests the  estimator 
\[
p_{est} = \sqrt{ 3 N_{K_4}/ N_{C_4}} \, \I(N_{C_4}>0).
\]
Indeed, it follows by our results on $N_{K_4}$ and $N_{C_4}$, and an application of the continuous mapping theorem, that $p_{est} = p + o_{\mathbb{P}}(1)$, i.e., $p_{est}$ is consistent. We observe that, conveniently, this estimator does not require any information about $m_n$ or $X_n$, as long as they satisfy the assumptions of Theorem \ref{the:sizeconverges}.

\section{Proofs}

We outline the structure of our proofs as follows. 
The proof of Theorem \ref{thm:mainexpectedvalue} consists of bounding $\E N_R(G_n)$ from above and below. The lower bound follows by a simple argument using the independence of the layers, and the upper bound follows from the union bound and the inequalities in Lemmas \ref{the:GoodCover} and \ref{the:UpperBoundNew}. Lemma \ref{the:UpperBoundNew}, which gives Bonferroni bounds for subgraph probabilities, is a central tool in our analysis and will be repeatedly used in the later proofs.

Theorem \ref{thm:mainvariance} is first proved for the case of bounded layer sizes, and then in the general setting. The first case follows by Chebyshev's inequality, where the variance terms are bounded using the inequalities in Lemmas \ref{the:UpperBoundNew} and \ref{thm:disjointcliques}. 
The requirement of bounded layer sizes is then removed, employing the convergence result in Lemma \ref{the:Truncated2Nontruncated}. 

The proofs of Theorems \ref{the:piconverges} and \ref{the:sizeconverges} are based on Theorem \ref{thm:mainvariance}. In Lemmas \ref{lemma_2_communities}--\ref{lemma:ucycle} we obtain upper bounds for the subgraph probabilities, which together with Lemma \ref{the:UpperBoundNew} yield the expressions for the expected subgraph frequencies. In Lemmas \ref{lemma:r2moments} and \ref{lemma:suppim} we derive technical results to verify the last condition of Theorem \ref{thm:mainvariance}. The proofs  of Theorems \ref{the:piconverges} and \ref{the:sizeconverges} are then finished using standard convergence theorems in probability theory.

\subsection{Edge covers}

A \emph{cover} of $E(R)$ is a collection $\cE$ of nonempty subsets $E \subset E(R)$ such that $\cup_{E \in \cE} E = E(R)$. Recall that $\abs{E}$ denotes the number of edges in $E$, and $\norm{E}$ the number of nodes incident to an edge in $E$.

\begin{lemma}
\label{the:GoodCover}
Let $R$ be a connected graph, and let $\cE$ be a cover of $E(R)$. Then
\begin{equation}
\label{eq:firstcover}
 \sum_{E \in \cE} (\norm{E} - 1) \wge \abs{V(R)} - 1. 
\end{equation}
If $\cE$ contains an overlapping pair of sets, then 
\begin{equation}
 \label{eq:GoodCover}
 \sum_{E \in \cE} (\norm{E} - 1) \wge \abs{V(R)},
\end{equation}
and
\begin{equation}
 \label{eq:GoodCoverSimple}
 \sum_{E \in \cE} \abs{E} \wge \abs{E(R)} +1.
\end{equation}
\end{lemma}
\begin{proof}
If $|\cE| = 1$, then $\cE=\{E(R)\}$ and trivially $\sum_{E\in \cE} (\norm{E}-1) = V(R)-1$. Let $t:= |\cE| \geq 2$. Denote by $V(E) = \cup_{e \in E} e$ the set of nodes incident to an edge in $E$. Because $R$ is connected, we may label the sets of $\cE$ as $E_1, \dots, E_t$ so that $V(E_s)$ overlaps $W_{s-1} = V(E_1) \cup \cdots \cup V(E_{s-1})$ for all $2 \le s \le t$. Denoting $V_s = V(E_s)$, we now find that
\[
 \big\lvert \bigcup_{s=1}^t V_s \big\rvert
 \weq \abs{V_1} + \sum_{s=2}^t \abs{V_s \setminus W_{s-1}}
 \weq \sum_{s=1}^t \abs{V_s} - \sum_{s=2}^t \abs{V_s \cap W_{s-1}}
\]
Because $\bigcup_{s=1}^t V_s = V(R)$ and $\sum_{s=1}^t \abs{V_s} = \sum_{E \in \cE} \norm{E}$, it follows that
\[
 \sum_{E \in \cE} \norm{E}
 \weq \abs{V(R)} + \sum_{s=2}^t \abs{V_s \cap W_{s-1}} \wge |V(R)| + t-1,
\]
which gives the first inequality of the claim.

Assume now that $\cE$ contains an overlapping pair.
We label the sets in $\cE$ as $E_1, \ldots, E_t$ so that $\abs{V(E_1 \cap E_2)} \ge 2$ and  $V(E_s)$ overlaps $W_{s-1} = V(E_1) \cup \cdots \cup V(E_{s-1})$ for all $3 \le s \le t$. 
Again we obtain
\[
 \sum_{E \in \cE} \norm{E}
 \weq \abs{V(R)} + \sum_{s=2}^t \abs{V_s \cap W_{s-1}}.
\]
Inequality \eqref{eq:GoodCover} follows by noting that
\[
 \sum_{s=2}^t \abs{V_s \cap W_{s-1}}
 \weq \abs{V_1 \cap V_2} + \sum_{s=3}^t \abs{V_s \cap W_{s-1}}
 \wge 2 + (t-2)
 \weq \abs{\cE}.
\]
Denote $F_s = E_1 \cup \cdots \cup E_s$. Inequality \eqref{eq:GoodCoverSimple} follows by
\[
 \abs{E(R)}
 \weq \big\lvert \bigcup_{s=1}^t E_s \big\rvert
 \weq \sum_{s=1}^t \abs{E_s} - \sum_{s=2}^t \abs{E_s \cap F_{s-1}}
 \wle \sum_{s=1}^t \abs{E_s} - \abs{E_1 \cap E_2},
\]
where $\abs{E_1 \cap E_2} \ge 1$.
\end{proof}

The following example shows that the assumption about the existence of an overlapping pair cannot be relaxed.
\begin{example*}
Let $R$ the a star graph with $r$ edges. Observe that $\norm{E} = \abs{E} + 1$ for any $E \subset E(R)$.
Hence $\sum_{E \in \cE} (\norm{E} - 1) = \sum_{E \in \cE} \abs{E}$. If the sets of $\cE$ are disjoint, then $\sum_{E \in \cE} (\norm{E} - 1) = \abs{E(R)} = \abs{V(R)} - 1$.
\end{example*}

\begin{lemma}
\label{the:DisjointGraphCover}
Let $R_1$ and $R_2$ be connected graphs such that $V(R_1) \cap V(R_2) = \emptyset$. Then for any partition $\cE$ of $E(R_1) \cup E(R_2)$, 
\begin{align*}
 \sum_{E \in \cE} (\norm{E} - 1)
 \wge \abs{V(R_1)} + \abs{V(R_2)} - 2 + \abs{\cE_0},
\end{align*}
where $\cE_0$ is the collection of sets in $\cE$ which overlap both $E(R_1)$ and $E(R_2)$.
\end{lemma}
\begin{proof}
We create a new partition of $E(R_1) \cup E(R_2)$ by replacing each $E \in \cE_0$ by two sets $E \cap E(R_1)$ and $E \cap E(R_2)$. The resulting partition can be written as
\[
 \cE' \weq (\cE \setminus \cE_0) \cup \cE_1 \cup \cE_2,
\]
where $\cE_i = \{ E \cap E(R_i): E \in \cE_0\}$ for $i=1,2$. Then $\abs{\cE'} = \abs{\cE} + \abs{\cE_0}$, and
\begin{align*}
 \sum_{E \in \cE'} \norm{E}
 &\weq \sum_{E \in \cE} \norm{E} - \sum_{E \in \cE_0} \norm{E}
 + \sum_{E \in \cE_1} \norm{E} + \sum_{E \in \cE_2} \norm{E} \\
 &\weq \sum_{E \in \cE} \norm{E}
 + \sum_{E \in \cE_0} \Big( \norm{E \cap E(R_1)} + \norm{E \cap E(R_2)} - \norm{E} \Big).
\end{align*}
Observe next that $V(E) = V(E \cap E(R_1)) \cup V(E \cap E(R_2))$. Because $V(R_1)$ and $V(R_2)$ are disjoint, so are $V(E \cap E(R_1))$ and $V(E \cap E(R_2))$, and it follows that $\norm{E} = \norm{E \cap E(R_1)} + \norm{E \cap E(R_2)}$. Hence we see that
\begin{align*}
 \sum_{E \in \cE'} \norm{E}
 &\weq \sum_{E \in \cE} \norm{E},
\end{align*}
and
\begin{align*}
 \sum_{E \in \cE} (\norm{E} - 1)
 \weq \sum_{E \in \cE'} (\norm{E} - 1) + \abs{\cE_0}.
\end{align*}
Next, we note that $\cE'$ can be written as a disjoint union $\cE' = \cE'_1 \cup \cE'_2$ in which $\cE'_i$ is a partition of $E(R_i)$. Because $R_i$ is connected, it follows by Lemma~\ref{the:GoodCover} that $\sum_{E \in \cE'_i} (\norm{E} - 1) \ge \abs{V(R_i)} - 1$. Using the above equality we conclude that
\begin{align*}
 \sum_{E \in \cE} (\norm{E} - 1)
 &\weq \sum_{E \in \cE'_1} (\norm{E} - 1) + \sum_{E \in \cE'_2} (\norm{E} - 1) + \abs{\cE_0} \\
 &\wge \abs{V(R_1)} + \abs{V(R_2)} - 2 + \abs{\cE_0}.
\end{align*}
\end{proof}

\subsection{Upper bounds related to subgraph probabilities}

Let $R$ be a graph with node set contained in $\{1,\dots,n\}$. Then $G_n \supset R$ if and only if $\cup_k E(G_{n,k}) \supset E(R)$, which means that for every $e \in E(R)$ there exists a layer $k$ such that $e \in E(G_{n,k})$. In other words, there is a mapping $\phi \colon E(R) \to \{1,\dots, m\}$ that assigns each edge $e \in E(R)$ some layer $\phi(e)$ which generates the edge (such $\phi$ can be seen as an edge coloring of $R$). Therefore,
\[
 \{ G_n \supset R \}
 \weq \bigcup_{\phi} A_\phi,
\]
where the union is over all $\phi \colon E(R) \to \{1,\dots, m\}$ and, denoting by $\phi^{-1}(k)$ the preimage of $\{k\}$ under $\phi$,
\[
 A_\phi
 \weq \bigcap_{e \in E(R)} \big\{ E(G_{n,\phi(e)}) \ni e \big\}
 \weq \bigcap_{k \in \ran(\phi)} \big\{E(G_{n,k}) \supset \phi^{-1}(k) \big\}
\]
is the event that for every $k$, the edge set $\phi^{-1}(k)$ is covered by layer $G_{n,k}$.
Denote the probability of $R$ being generated by $G_n$ by $f(R) = \pr( G_n \supset R)$.
Bonferroni's inequalities imply that $U(R) - \frac12 L(R) \le f(R) \le U(R)$, where
\[
 U(R)
 \weq \sum_{\phi} \pr(A_\phi)
 \qquad \text{and} \qquad
 L(R)
 \weq \sumd_{\phi, \psi} \pr(A_\phi, A_\psi),
\]
where the first sum is over all $\phi \colon E(R) \to \{1,\dots, m\}$, and $\sum'$ denotes a sum over pairs with $\phi \neq \psi$.
The following result gives explicit bounds for $U(R)$ and $L(R)$.

\begin{lemma}
\label{the:UpperBoundNew}
Let $R$ be a connected graph with $r$ nodes and $s$ edges such that $V(R) \subset \{1,\dots, n\}$. For any $x \in (0,\infty)$, $y \in (0, \infty)$, and $n \ge x \vee r$ it holds that
\begin{align}
\label{eq:UpperBoundNew}
 \sum_\phi \pr(A_\phi) &\wle \left( e^{r!c x m/n} - 1 \right) s^s x^{r-1} y^s n^{1-r},
\end{align}
 and if $y \in (0,1]$,
\begin{align} \label{eq:LowerBoundNew}
 \sumd_{\phi, \psi} \pr(A_\phi, A_\psi)
 &\wle \left( e^{r! c x m/n} - 1 \right) (2s)^{2s} x^r y^{s+1} n^{-r},
\end{align}
where $c = \max_{E \subset E(R)} \frac{ (\pi)_{\norm{E}, \abs{E}}}{x^\norm{E} y^\abs{E}}$.
\end{lemma}
\begin{proof}
Fix $\phi \colon E(R) \to \{1,\dots, m\}$. Because layers are mutually independent, we find that
\begin{equation*}
 \pr(A_\phi)
 \weq \prod_{k \in \ran(\phi)} \frac{(\pi)_{||\phi^{-1}(k)||, \abs{\phi^{-1}(k)}}}{(n)_{||\phi^{-1}(k)||}}.
\end{equation*}
Denote $K = \ran(\phi)$ and $E_k = \phi^{-1}(k)$. Then the inequalities $\norm{E_k} \le r$ and $(n)_a \ge n^a/a!$ for $n \ge a$ imply that
\begin{equation*}
 \pr(A_\phi)
 \wle \prod_{k \in K} \frac{ \norm{E_k} ! c x^\norm{E_k} y^\abs{E_k}}{n^{\norm{E_k}}}
 \wle (r!c)^{\abs{K}} \, (x/n)^{\sum_{k} \norm{E_k} } y^{\sum_k \abs{E_k}},
\end{equation*}
Now \eqref{eq:firstcover} in Lemma~\ref{the:GoodCover} implies that $\sum_{k \in K} \norm{E_k} \ge r + \abs{K}-1$, and we also see that $\sum_{k \in K} \abs{E_k} = s$ because the sets $E_k$ form a partition of $E(R)$.  Hence, for $n \ge x$,
\[
 \pr(A_\phi)
 \wle x^{r-1} y^s n^{1-r}  (r!c x/n)^{\abs{K}}.
\]
Observe that the number of $\phi \colon E(R) \to \{1,\dots, m\}$ with $\abs{\ran(\phi)} = t$ is bounded by $\binom{m}{t} t^s$. Therefore,
\[
 \sum_\phi \pr(A_\phi)
 \wle x^{r-1} y^s n^{1-r} \sum_{t=1}^s \binom{m}{t} t^s (r!c x/n)^t,
\]
and since
\[
 \sum_{t=1}^s \binom{m}{t} t^s (r!c x/n)^t
 \wle s^s \sum_{t=1}^s \frac{(r!c x m/n)^t}{t!}
 \wle s^s \left( e^{r!c x m/n} - 1 \right),
\]
the upper bound of \eqref{eq:UpperBoundNew} follows.

(ii) Let us verify \eqref{eq:LowerBoundNew}.  Let $\phi, \psi: E(R) \to \{1,\dots, m\}$ be distinct. We find that
\[
 A_\phi \cap A_\psi
 \weq \bigcap_{k \in K} \big\{E(G_{k}) \supset E_k \big\}.
\]
where $K = \ran(\phi) \cup \ran(\psi)$ and $E_k = \phi^{-1}(k) \cup \psi^{-1}(k)$. The independence of the layers implies that
\[
 \pr(A_\phi, A_\psi)
 \weq \prod_{k \in K} \frac{(\pi)_{\norm{E_k}, \abs{E_k}}}{(n)_{\norm{E_k}}}.
\]
Then the inequalities $\norm{E_k} \le r$ and $(n)_a \ge n^a/a!$ for $n \ge a$ imply that
\begin{equation}
 \pr(A_\phi, A_\psi)
 \wle \prod_{k \in K} \frac{ \norm{E_k} ! c x^\norm{E_k} y^\abs{E_k} }{n^{\norm{E_k}}}
 \wle (r!c)^{\abs{K}} \, (x/n)^{\sum_{k} \norm{E_k} } y^{\sum_k \abs{E_k}}.
\end{equation}
Since $\phi \ne \psi$, there exists an edge $e \in E(R)$ such that $\phi(e) \ne \psi(e)$. Denote $k' = \phi(e)$ and $k'' = \psi(e)$. Then $e \in E_{k'} \cap E_{k''}$, and $\cE = \{E_k: k \in K\}$ is a cover of $E(R)$ containing an overlapping pair of sets. Hence by Lemma~\ref{the:GoodCover}, it follows that $\sum_{k \in K} \norm{E_k} \ge r + \abs{K}$. The same lemma also implies that $\sum_{k \in K} \abs{E_k} \ge  s + 1$. Hence,
\[
 \pr(A_\phi, A_\psi)
 \wle (r!c)^{\abs{K}} \, (x/n)^{r + \abs{K} } y^{s+1}.
\]
By noting that the number of distinct pairs of $(\phi, \psi)$ is bounded by $\binom{m}{t}t^{2s}$, it follows that
\[
 \sumd_{\phi, \psi} \pr(A_\phi, A_\psi)
 \wle x^r y^{s+1} n^{-r} \sum_{t=1}^{2s} \binom{m}{t} t^{2s} \left( \frac{r! c x}{n} \right)^t.
\]
Inequality \eqref{eq:LowerBoundNew} now follows by
\[
 \sum_{t=1}^{2s} \binom{m}{t} t^{2s} \left( \frac{r! c x}{n} \right)^t
 \wle (2s)^{2s} \sum_{t=1}^{2s} \frac{(r! c x m/n)^t}{t!}
 \wle (2s)^{2s} \left( e^{r!c x m/n} - 1 \right).
\]
\end{proof}

\subsection{Analysis of expected frequencies}

\begin{proof}[Proof of Theorem~\ref{thm:mainexpectedvalue}]
Denote $q_n = (\frac{m_n}{n} (\pi_n)_{r,s})^{1/s}$.
By applying the union bound $f(R) \le U(R)$, and Lemma \ref{the:UpperBoundNew} with $x=1, y=q_n$, we find that 
\[
 U(R)
 \wle (e^{r!cm_n/n} -1) s^s q_n^s n^{1-r}
\]
where $c = \max_{E \subset E(R)} \frac{(\pi_n)_{\norm{E},|E|}}{q_n^{\abs{E}}}$.
Assumption~\eqref{eq:maxcond} in this case implies that $\left( \frac{m_n}{n} (\pi_n)_{\norm{E},|E|} \right)^{1/\abs{E}} \lesim q_n$ for all $E \subset E(R)$, so that
\begin{align*}
 c
 \weq \frac{n}{m_n} \max_{E \subset E(R)}
 \left\{ \frac{1}{q_n} \Big( \frac{m_n}{n} (\pi_n)_{\norm{E},|E|} \Big)^{1/\abs{E}} \right\}^\abs{E}
 \wlesim \frac{n}{m_n},
\end{align*}
and so $U(R) \lesim q_n^s n^{1-r} = m_n n^{-r} (\pi_n)_{r,s}$.

For the other direction, we note that by independence,
\[
 \pr \Big( \bigcup_{ \! \phi: |\ran(\phi)|=1} \cA_\phi \Big)
 \weq 1 - \nquad \prod_{\phi: |\! \ran(\phi)|=1} \nhquad \Big( 1-\pr(\cA_\phi ) \Big)
 \weq 1- \Big( 1-\frac{(\pi_n)_{r,s}}{(n)_r} \Big)^m.
\]
We note that $1-(1-p)^m \ge mp - \frac12 (mp)^2 = (1-o(1))mp$ for $mp \ll 1$, and
$1-(1-p)^m \asymp 1$ for $mp \gesim 1$. We conclude that
\[
 f(R)
 \wge \pr \Big( \bigcup_{ \! \phi: |\ran(\phi)|=1} \cA_\phi \Big)
 \wgesim m_n \frac{(\pi_n)_{r,s}}{(n)_r} \wedge 1.
\]
Being a probability, it obvious that $f(R) \le 1$.  Therefore, we conclude that
$f(R) \asymp m_n \frac{(\pi_n)_{r,s}}{(n)_r} \wedge 1$. 
The claim now follows by noting that $\E N_R(G_n) =  \frac{(n)_r}{\abs{\Aut(R)}} f(R)$ where $\abs{\Aut(R)}$ is the number of graph automorphisms of $R$.
\end{proof}

\subsection{Analysis of variances}

\begin{lemma}
\label{thm:disjointcliques}
For any connected graphs $R_1$ and $R_2$ having disjoint node sets contained in $\{1,\dots, n\}$, and any 
$x \in (0,\infty)$, $y \in (0,1]$, and $n \ge x \vee r$,
\begin{equation}
 \label{eq:DisjointGraphs}
 \begin{aligned}
 &\pr( G \supset R_1, G \supset R_2) - \pr(G \supset R_1) \pr( G \supset R_2) \\
 &\qquad \wle e^{r!c x m/n} \left( 1 + (2s)^{2s} \right) s^s x^{r -1} y^s  n^{1-r},
 \end{aligned}
\end{equation}
where $c = \max_{E \subset E(R_1) \cup E(R_2)} \frac{ (\pi)_{\norm{E}, \abs{E}}}{x^\norm{E} y^\abs{E}}$, 
 $r = \abs{V(R_1)} + \abs{V(R_2)}$, and $s = \abs{E(R_1)} + \abs{E(R_2)}$.
\end{lemma}

\begin{proof}
(i) By Bonferroni's inequalities
\begin{align*}
 f(R_1) f(R_2)
 &\wge \Big( U(R_1) - \frac12 L(R_1) \Big) \Big(  U(R_2) - \frac12 L(R_2) \Big) \\
 &\wge U(R_1)U(R_2) - \frac12 \Big( U(R_1) L(R_2) + U(R_2) L(R_1) \Big).
\end{align*}
Since $\pr(G \supset R_1, G \supset R_2) = f(R_1 \cup R_2) \le U(R_1 \cup R_2)$, it follows that the left-hand side of \eqref{eq:DisjointGraphs} is bounded from above by
\begin{equation}
 \label{eq:PartialBound}
 \begin{aligned}
 f(R_1 \cup R_2) - f(R_1) f(R_2)
 &\wle U(R_1 \cup R_2) - U(R_1) U(R_2) \\
 &\qquad + \frac12 \Big( U(R_1) L(R_2) + U(R_2) L(R_1) \Big).
\end{aligned}
\end{equation}

(ii) Let us derive an upper bound for $U(R_1 \cup R_2)$. Denote by $\cF$ the set of functions $\phi \colon E(R_1) \cup E(R_2) \to \{1,\dots, m\}$, and denote by $\phi_i$ the restriction of $\phi \in \cF$ to $E(R_i)$. We partition $\cF = \cF^\perp \cup \cF^\parallel$, where $\cF^\perp$ denotes the set of functions $\phi \in \cF$ such that $\ran(\phi_1) \cap \ran(\phi_2) = \emptyset$, and $\cF^\parallel = \cF \setminus \cF^\perp$. Then for any $\phi \in \cF^\perp$, we find that
\[
 \pr(A_\phi)
 \weq \left( \prod_{k \in \ran(\phi_1)} \nquad \frac{(\pi)_{||\phi^{-1}(k)||, \abs{\phi^{-1}(k)}}}{(n)_{||\phi^{-1}(k)||}} \right)
 \left( \prod_{k \in \ran(\phi_2)} \nquad \frac{(\pi)_{||\phi^{-1}(k)||, \abs{\phi^{-1}(k)}}}{(n)_{||\phi^{-1}(k)||}} \right),
\]
from which we conclude that $\pr(A_\phi) = \pr(A_{\phi_1}) \pr(A_{\phi_2})$. Therefore,
\[
 \sum_{\phi \in \cF^\perp } \pr(A_\phi)
 \weq \sum_{\phi \in \cF^\perp } \pr(A_{\phi_1}) \pr(A_{\phi_2})
 \weq U(R_1) U(R_2).
\]
Consider next a function $\phi \in \cF^\parallel$.  Denote $K = \ran(\phi)$, $E_k = \phi^{-1}(k)$, $r_i = |V(R_i)|$, and $s_i = |E(R_i)|$. Then the inequalities $\norm{E_k} \le r_1+r_2$ and $(n)_a \ge n^a/a!$ for $n \ge a$ imply that
\[
 \pr(A_\phi)
 \wle \prod_{k \in K} \frac{ \norm{E_k} ! c x^\norm{E_k} y^\abs{E_k}}{n^{\norm{E_k}}}
 \wle ((r_1+r_2)!c)^{\abs{K}} \, (x/n)^{\sum_{k} \norm{E_k} } y^{\sum_k \abs{E_k}}.
\]
By Lemma \ref{the:DisjointGraphCover}, $\sum_{k \in K} \norm{E_k} \ge r_1 + r_2 + \abs{K} - 1$. Furthermore, $\sum_k \abs{E_k} = s_1 + s_2$. Hence,
\[
 \pr(A_\phi)
 \wle ((r_1+r_2)!c)^{\abs{K}} \, (x/n)^{r_1+r_2 + \abs{K}-1 } y^{s_1 + s_2},
\]
Now denote by $\cF_t$ the set of $\phi \in \cF$ such that $\abs{\ran(\phi)} = t$. Then we find that $\abs{\cF^\parallel \cap \cF_t} \le \abs{\cF_t} \le \binom{m}{t} t^{s_1+s_2}$. Denoting $r = r_1 + r_2$ and $s = s_1 + s_2$, it follows that
\begin{align*}
 \sum_{\phi \in \cF^\parallel} \pr(A_\phi)
 &\wle \sum_{t=1}^{s} \binom{m}{t} t^{s} (r!c)^t \, (x/n)^{r +t -1 } y^{s} \\
 &\wle s^{s} (x/n)^{r -1} y^{s} \sum_{t=1}^{s} \binom{m}{t}  (r!c)^t \, (x/n)^t \\
 &\wle s^{s} (x/n)^{r -1} y^{s} \sum_{t=1}^{s} \frac{(r!c x m/n)^t}{t!}.
\end{align*}
Hence,
\[
 \sum_{\phi \in \cF^\parallel} \pr(A_\phi)
 \wle \left( e^{r!c x m/n} - 1 \right) s^s x^{r -1} y^s  n^{1-r}.
\]
We conclude that
\begin{equation}
\label{eq:bonferroni1}
 U(R_1 \cup R_2) - U(R_1) U(R_2)
 \wle \left( e^{r!c x m/n} - 1 \right) s^s x^{r -1} y^s  n^{1-r}.
\end{equation}

(iii) Let us derive an upper bound for the second term on the right-hand side of \eqref{eq:PartialBound}. Denoting $c_i = \max_{E \subset E(R_i)} \frac{ (\pi)_{\norm{E}, \abs{E}}}{x^\norm{E} y^\abs{E}}$, Lemma~\ref{the:UpperBoundNew} implies that 
\begin{align*}
U(R_i) &\wle ( e^{r_i! c_i x m/n} - 1 ) s_i^{s_i}  x^{r_i-1} y^{s_i} n^{1-r_i}, \quad \text{and} \\
L(R_i) &\wle ( e^{r_i! c_i x m/n} - 1 ) (2s_i)^{2s_i} x^{r_i} y^{s_i+1} n^{-r_i}. 
\end{align*}
Therefore,
\begin{equation}
\label{eq:bonferroni2}
  \frac12 \Big( U(R_1) L(R_2) + U(R_2) L(R_1)  \Big)
 \wle e^{r! c x m/n} s^s (2s)^{2s} x^{r-1} y^{s+1} n^{1-r},
\end{equation}
where we used $c_1, c_2 \le c$ and $r_1! + r_2! \le r!$.

(iv) Combining \eqref{eq:PartialBound} with \eqref{eq:bonferroni1} and \eqref{eq:bonferroni2} gives
\[
  f(R_1 \cup R_2) - f(R_1) f(R_2)
 \wle e^{r!c x m/n} \left( 1 + (2s)^{2s} y \right) s^s x^{r -1} y^s  n^{1-r},
\]
and the claim follows by $y \le 1$.
\end{proof}

\subsection{Proof of Theorem~\ref{thm:mainvariance}}

\begin{proof}[Proof of Theorem~\ref{thm:mainvariance}]

In the following we denote $q_n = (\pi_n)_{r,s}^{1/s}$.

(i) Let us first make an extra assumption that the support of $\pi_n$ is contained in $\{0,\dots, M\} \times [0,1]$ for all $n$, where $M \ge 1$ is some fixed integer. Then $(\pi_n)_{a,b} \le M^a (\pi_n)_{0,b}$, and
\[
 \frac{(\pi_n)_{\norm{E},\abs{E}}}{M^\norm{E} q_n^\abs{E}}
 \wle \max_{1 \le b \le 2s} \frac{(\pi_n)_{0,b}}{q_n^b}
 \ =: \ d'_n
\]
for all $n$ and for any set $E$ of node pairs from $\{1,\dots,n\}$.

(ii) We start by deriving an upper bound for the variance of $N_R = N_R(G_n)$. By writing $N_R = \sum_{R' \in \Sub (R,K_n)} \ind(G_n \supset R')$, $\Var( N_R ) = \Cov(N_R, N_R)$, and applying the bilinearity of the covariance operator, we find that $\Var(N_{R}) = \sum_{t=0}^r V_{n,t}$, where
\begin{align*}
 V_{n,t} \weq \sum_{(R', R'') \in \cR_{n,t}} \bigg( f(R' \cup R'') - f(R') f(R'') \bigg),
\end{align*}
where $\cR_{n,t}$ denotes the set of pairs $(R',R'')$ of distinct $R$-isomorphic graphs such that $V(R'), V(R'') \subset \{1,\dots,n\}$ and $\abs{V(R') \cap V(R'')} = t$.

Let $t=0$. Since $(\pi_n)_{\norm{E}, |E|} \le M^{\norm{E}}$, we may normalize $(q_n)$ to be in $(0,1]$. Lemma \ref{thm:disjointcliques} applied with $x=M$ and $y = q_n$ implies that for any $(R',R'') \in \cR_{n,0}$,
\[
 f(R' \cup R'') -  f(R')f(R'')
 \wle e^{(2r)! d_{0,n} M m_n/n} \left( 1 + (4s)^{4s} \right) (2s)^{2s} M^{2r-1} q_n^{2s} n^{1-2r}.
\]
where $d_{0,n} = \max_{E \subset E(R') \cup E(R'')} \frac{ (\pi_n)_{\norm{E}, \abs{E}}}{M^\norm{E} q_n^\abs{E}}$.  The extra assumption of the layer sizes being bounded by $M$ implies that $d_{0,n} \le d_{n}'$.
Assumption $(\pi_n)_{0,2s} \lesim (\pi_n)_{r,s}^2$ together with Jensen's inequality further implies that
$(\pi_n)_{0,b}^{1/b} \lesim q_n$ for all $1 \le b \le 2s$, so that $d_n' \lesim 1$.
Hence $d_{0,n} \lesim 1$.
Because $\abs{\cR_{n,0}} \le ( \frac{(n)_r}{\abs{\Aut(R)}} )^2 \le n^{2r}$
and $m_n \lesim n$, we conclude that $V_{n,0} \lesim q_n^{2s} n$.

For $t \ge 1$, let $(R',R'') \in \cR_{n,t}$. The graph $R' \cup R''$ is then connected and such that $\abs{V(R' \cup R'')} = 2r-t$ and $s \le \abs{E(R' \cup R'')} \le 2s$. Lemma~\ref{the:UpperBoundNew} applied to  $R' \cup R''$ with $x = M$, $y = q_n$  shows that
\[
 f(R' \cup R'') \wle \left( e^{(2r-t)!d_{t,n} M m_n/n} - 1 \right) (2s)^{2s} M^{2r-t-1} q_n^{s} n^{1-2r+t},
\]
where $d_{t,n} = \max_{E \subset E(R') \cup E(R'')} \frac{ (\pi_n)_{\norm{E}, \abs{E}}}{M^\norm{E} q_n^\abs{E}}$ is again bounded by $d_{t,n} \le d'_n$. Because $d'_n \lesim 1$ and $\abs{\cR_{n,t}} \le \binom{n}{2r-t} (r!)^2 \le (r!) n^{2r-t}$, it follows that $V_{n,t} \lesim q_n^{s} n$ for all $1 \le t \le r$. Hence
$\Var(N_R) \lesim q_n^{s} n$.

Next, Theorem \ref{thm:mainexpectedvalue} confirms that $\E N_{R}(G_n) \asymp n (\pi_n)_{r,s} \asymp q_n^{s} n$, so that
\[
 \frac{\Var(N_{R}(G_n))}{( \E N_R(G_n) )^2} 
 \wlesim \frac{q_n^{s} n}{\big( q_n^{s} n \big)^2}
 \wasymp \frac{1}{n (\pi_n)_{r,s} }
 \wll 1
\]
due to assumption $(\pi_n)_{r,s} \gg n^{-1}$, and the claim follows by applying Chebyshev's inequality to
the random variable $\frac{{N_R - \E N_R}}{\E N_R}$.

(iii) Let us continue the proof, now without the extra assumption of bounded layer sizes. Let $S_n = N_R(G_n)$ be the number of $R$-isomorphic subgraphs of $G_n$, and let $S_{n,M}$ be the corresponding count for the layer-truncated graph
\[
V(G_n^M) = [n], \qquad E(G_{n}^M) = \bigcup_{k: \, X_{n,k} \le M} E(G_{n,k}).
\]
Then $S_n - S_{n,M}$ counts the $R$-isomorphic subgraphs of $G_n$ that vanish when layers larger than $M$ are removed.  Exchangeability implies that
\[
 \E (S_n - S_{n,M})
 \weq \abs{\cG_n(R)} \,\, \pr( G_n \supset R, \, \Gnm \not\supset R),
\]
where $\cG_n(R)$ denotes the set of $R$-isomorphic graphs with node set contained in $\{1,\dots, n\}$. We will first verify that
\begin{equation}
 \label{eq:TruncatedUnionKey}
 \pr( G_n \supset R, \, \Gnm \not\supset R)
 \wle s^s \hat{c}_{M,n} q_n^s n^{1-r} (r! m_n/n) e^{r! c_n m_n/n},
\end{equation}
where
\[
 c_n = \max_{E \subset E(R)} \frac{(\pi_n)_{\norm{E}, \abs{E}}}{q_n^\abs{E}}
 \qquad \text{and} \qquad
 \hat{c}_{M,n} = \max_{E \subset E(R)} \frac{(\pi_n^M)_{\norm{E}, \abs{E}}}{q_n^\abs{E}},
\]
and we recall the definition of $(\pi_n^M)_{\norm{E}, \abs{E}}$ in \eqref{eq:pinmdef}.

The union bound implies that
\begin{equation}
 \label{eq:TruncatedUnion0}
 \pr( G_n \supset R, \Gnm \not\supset R)
 \wle \sum_\phi \pr( A_\phi, \Gnm \not\supset R),
\end{equation}
where the sum is over all $\phi\!:\! E(R) \to \{1,\dots,m\}$.
On the event $A_\phi \cap \{ \Gnm \not\supset R\}$, 
there exists $k' \in \ran(\phi)$ such that $X_{n,k'} > M$. Therefore,
\begin{equation}
 \label{eq:TruncatedUnion1}
 \pr( A_\phi, \, \Gnm \not\supset R)
 \wle \sum_{k' \in \ran(\phi)} \pr(A_\phi, \, X_{n,k'} > M).
\end{equation}
Denote $E_k = \phi^{-1}(k)$ and $K = \ran(\phi)$. The independence of the layers then implies that
\begin{align*}
 \pr(A_\phi, \, X_{n,k'} > M)
 &\weq \frac{(\pi^M_n)_{\norm{E_{k'}}, \abs{E_{k'}}}}{(n)_{\norm{E_{k'}}}}
 \prod_{k \in K \setminus \{k'\}} \nhquad \frac{(\pi_n)_{\norm{E_{k}}, \abs{E_{k}}}}{(n)_{\norm{E_{k}}}} \\
 &\wle \frac{\norm{E_{k'}}! \hat{c}_{M,n} q_n^\abs{E_{k'}}}{n^{\norm{E_{k'}}}}
 \prod_{k \in K \setminus \{k'\}} \frac{\norm{E_{k}}! c_n q_n^\abs{E_k}}{n^{\norm{E_{k}}}} \\
 &\wle \hat{c}_{M,n} c_n^{\abs{K}-1} (r!)^\abs{K}   n^{- \sum_k \norm{E_k}} q_n^{\sum_k \abs{E_k}}.
\end{align*}
By Lemma \ref{the:GoodCover}, $\sum_k \norm{E_k} \ge r-1 + \abs{K}$, and moreover, $\sum_k \abs{E_k} = s$. These together with \eqref{eq:TruncatedUnion1} give
\begin{align*}
 \pr(A_\phi, \, \Gnm \not\supset R)
 &\wle \hat{c}_{M,n} c_n^{\abs{K}-1}  n^{1-r} \abs{K} (r!)^\abs{K}   n^{-\abs{K}} q_n^s.
\end{align*}
By noting that the number of $\phi \colon E(R) \to \{1,\dots,m\}$ with range of size $t$ is bounded by $\binom{m}{t} t^s$, it follows by  \eqref{eq:TruncatedUnion0} that
\begin{align*}
 \pr( G_n \supset R, \, \Gnm  \not\supset R)
 &\wle \sum_{t=1}^s \binom{m}{t} t^s \hat{c}_{M,n} c_n^{t-1} q_n^s n^{1-r} t (r!)^t   n^{-t} \\
 &\wle s^s \hat{c}_{M,n} q_n^s n^{1-r} (r! m/n) \sum_{t=1}^s \frac{(r! c_n m/n)^{t-1}}{(t-1)!},
\end{align*}
from which \eqref{eq:TruncatedUnionKey} follows.

Next, recall from Theorem~\ref{thm:mainexpectedvalue} that
$\pr( G_n \supset R) \asymp n^{1-r} (\pi_n)_{r,s}$.
By combining this with \eqref{eq:TruncatedUnionKey}, we see that there exist positive constants $n_0, a_1,a_2$ (not depending on $M$ or $n$) such that
\begin{align*}
 \frac{\E (S_n - S_{n,M})}{\E S_n}
 \weq \frac{ \pr( G_n \supset R, \, \Gnm \not\supset R) } { \pr( G_n \supset R ) }
 \wle \frac{a_2 \hat{c}_{M,n} n^{1-r} (\pi_n)_{r,s}   } { a_1 n^{1-r} (\pi_n)_{r,s}}
\end{align*}
for all $n \ge n_0$. Therefore,
\begin{align*}
 \limsup_{n \to \infty} \frac{\E (S_n - S_{n,M})}{\E S_n}
 \wle (a_2/a_1) \sup_n \hat{c}_{M,n}
\end{align*}
for all $M$. Because  $\frac{S_{n,M}}{\E S_{n,M}} \prto 1$ for every $M$ (by parts (i) and (ii)), and $ \sup_n \hat{c}_{M,n} \to 0$ as $M \to \infty$, the general claim follows by Lemma~\ref{the:Truncated2Nontruncated} by setting $h(M) = \sup_n \hat{c}_{M,n}$ and noting that $S_n-S_{n,M} = |S_{n,M}-S_n|$.
\end{proof}

\paragraph{From truncated limits to nontruncated limits.}

\begin{lemma}
\label{the:Truncated2Nontruncated}
Let $S_n, S_{n,M} \ge 0$ be random variables such that for any $M \in \Z_+$, $\frac{S_{n,M}}{\E S_{n,M}} \prto 1$ as $n \to \infty$, and $\limsup_{n \to \infty} \frac{\E \abs{S_{n,M} - S_n}}{\E S_n} \le h(M)$ for some function $h(M)$ such that $\lim_{M \to \infty} h(M) = 0$. Then $\frac{S_{n}}{\E S_{n}} \prto 1$ as $n \to \infty$.
\end{lemma}
\begin{proof}
We may write
\begin{align*}
 \frac{S_{n}}{\E S_{n}} - 1
 &\weq \underbrace{\left( \frac{S_{n,M}}{\E S_{n,M}} -  1 \right)}_{A_{n,M}}
 - \underbrace{ \frac{S_{n,M} - S_n}{\E S_{n}} }_{B_{n,M}}
 + \underbrace{ \left( \frac{\E (S_{n,M} - S_n)}{\E S_{n}} \right) \frac{S_{n,M}}{\E S_{n,M}} }_{C_{n,M}}.
\end{align*}
Markov's inequality then implies that for any $\epsilon > 0$ and $M \ge 1$,
\begin{align*}
 \pr \Big( \big\lvert \frac{S_{n}}{\E S_{n}} - 1 \big\rvert > \epsilon \Big)
 &\wle \pr \Big( \abs{A_{n,M}} > \epsilon/3 \Big)
 + \pr \Big( \abs{B_{n,M}} > \epsilon/3 \Big)
 + \pr \Big( \abs{C_{n,M}} > \epsilon/3 \Big) \\
 &\wle \pr \Big( \abs{A_{n,M}} > \epsilon/3 \Big)
 + 2(\epsilon/3)^{-1} \Big( \E \abs{B_{n,M}} + \E \abs{C_{n,M}} \Big) \\
 &\weq \pr \Big( \abs{A_{n,M}} > \epsilon/3 \Big)
 + 4(\epsilon/3)^{-1} \E \abs{B_{n,M}}.
\end{align*}
Hence,
\begin{align*}
 \limsup_{n \to \infty} \pr \Big( \big\lvert \frac{S_{n}}{\E S_{n}} - 1 \big\rvert > \epsilon \Big)
 &\wle 4(\epsilon/3)^{-1} h(M).
\end{align*}
Because the above inequality holds for all $M$ and $\epsilon>0$, the claim follows.
\end{proof}

\subsection{Analysis of clique and cycle frequencies}
The following lemma gives a simpler alternative for the last condition of Theorem \ref{thm:mainvariance}.  For cycles and cliques, these simpler conditions are also necessary, as the conditions given in $(i)$ and $(ii)$ below correspond to specific edge subsets $E \subset E(R)$. 
This is not true for a general $R$. For example, consider a 5-star graph and condition $(i)$ below. Then $r = 6$ and $\norm{E} = \abs{E} + 1$ for all non-empty $E$. Since there does not exist $E$ such that $\norm{E} = r = 6$ and $\abs{E} = r/2 = 3$, condition $(i)$ becomes superfluous.

Lemma \ref{lemma:suppim} is analogous to a classical result on uniform integrability, but for joint distributions of two random variables. Lemmas \ref{lemma_2_communities}--\ref{lemma:ucycle} give the Bonferroni upper bounds for clique and cycle probabilities.
\begin{lemma}
\label{lemma:r2moments}
Let $R$ be a graph with $r$ nodes. Assume that $\pi_n \to \pi$ weakly. If either
\begin{enumerate}[(i)]
\item $r$ is even and
\[
\sup_n \, (\pi_n^M)_{r,r/2} \to 0 \quad \text{as} \quad M \to \infty, \text{ or}
\]
\item $r$ is odd and
\[
\sup_n \, (\pi_n^M)_{r-1,\frac{r-1}{2}} \to 0 \quad \text{and} \quad \sup_n \, (\pi_n^M)_{r,\frac{r+1}{2}}  \to 0 \quad \text{as} \quad  M \to \infty,
\]
\end{enumerate}
then
\begin{equation*}
 \sup_n \max_{E \subset E(R)} (\pi_n^M)_{\norm{E},|E|}
 \wto 0 \quad \text{as} \quad M \to \infty.
\end{equation*}
\end{lemma}
\begin{proof}
Let $(X_n,Y_n)$ be a random vector with distribution $\pi_n$, and let $E \subset E(R)$. Because the graph induced by $E$ has no isolated nodes, $|E| \geq  \norm{E}/2$, and
\begin{align*}
(X_n)_{\norm{E}} Y_n^{|E|} \le (X_n)_{\norm{E}} Y_n^{\norm{E}/2}.
\end{align*}
Since $r/\norm{E} \geq 1$,
\begin{align*}
(X_n)_{\norm{E}} Y_n^{\norm{E}/2} &\wle 1+ \I((X_n)_{\norm{E}} Y_n^{\norm{E}/2} > 1) (X_n)_{\norm{E}} Y_n^{\norm{E}/2} \\
&\wle 1+ \I((X_n)_{\norm{E}} Y_n^{\norm{E}/2} > 1) (X_n)_{\norm{E}}^{r/\norm{E}} Y_n^{r/2}.
\end{align*}
We split further
\begin{align*}
(X_n)_{\norm{E}} Y_n^{\norm{E}/2} \wle 1&+ \I\Big[ (X_n)_{\norm{E}} Y_n^{\norm{E}/2} \in (1,r) \Big] \,(X_n)_{\norm{E}}^{r/\norm{E}} Y_n^{r/2} \\
&+ \I\Big[ (X_n)_{\norm{E}} Y_n^{\norm{E}/2} \geq r \Big] \, (X_n)_{\norm{E}}^{r/\norm{E}} Y_n^{r/2},
\end{align*}
and by applying $(x)_{\norm{E}}^r \le (r! (x)_r)^{\norm{E}} $ for the last term, it follows that
\begin{equation}
\label{eq:dominate1}
(X_n)_{\norm{E}} Y_n^{|E|} \wle 1 + r!^{r/\norm{E}} + r!(X_n)_r Y_n^{r/2}.
\end{equation}
Multiplying both sides by $\I(X_n>M)$ and taking $\E$ and $\sup_n$ gives the upper bound
\[
\sup_n \, \Big( (1+ r!^{r/\norm{E}})\pr (X_n>M) + r!\E[(X_n)_r Y_n^{r/2}\I(X_n>M)] \Big).
\]
When $M \to \infty$, the first term goes to zero by the weak convergence of $\pi_n$, and the second term goes to zero by assumption. This establishes $(i)$. 

If $r$ is odd and $\norm{E} \le r-1$, the same argument as above gives
\begin{equation}
\label{eq:dominate2}
(X_n)_{\norm{E}} Y_n^{|E|} \wle 1 + (r-1)!^{\frac{r-1}{\norm{E}}} + (r-1)!(X_n)_{r-1} Y_n^{\frac{r-1}{2}}.
\end{equation}
If $\norm{E}=r$, it suffices to note that $|E| \geq \frac{r+1}{2}$, hence
\begin{equation}
\label{eq:dominate3}
(X_n)_{\norm{E}} Y_n^{|E|} \le (X_n)_{r} Y_n^{\frac{r+1}{2}}.
\end{equation}
In either case multiplying by  $\I(X_n>M)$ and taking $\E$ and $\sup$ gives that $\sup_n \E[\I(X_n >M) (X_n)_{\norm{E}} Y_n^{|E|}] \to 0$, i.e., $(ii)$.
\end{proof}
\begin{lemma}
\label{lemma:suppim}
If $\pi_n \to \pi$ weakly and $(\pi_n)_{a,b} \to (\pi)_{a,b} < \infty$, Then
\[
\sup_n (\pi_n^M)_{a,b} \to 0 \quad \text{as} \quad M \to \infty,
\]
where $(\pi_n^M)_{a,b}$ is defined in \eqref{eq:pinmdef}.
\end{lemma}
\begin{proof}
The claim is equivalent to $\lim_{M\to \infty} \limsup_n (\pi_n^M)_{a,b} = 0$ by a standard argument.
Let $(X_n,Y_n) \sim \pi_n$, and define
\[
h_M(x,y) = 
\begin{cases*} 
0 & if $x>M$ \\
(x)_a y^b & if $x<M/2$ \\
M-(x)_a y^b & if $x \in [M/2,M]$.
\end{cases*} 
\]
Then
\[
(X_n)_a Y_n^b \I (X_n > M) \wle (X_n)_a Y_n^b - h_M(X_n,Y_n).
\]
Since $h_M$ is continuous and bounded on $\mathbb{R}_+ \times [0,1]$, we have by weak convergence $\E[h_M(X_n,Y_n)] \to \E[h_M(X,Y)] < \infty$ as $n \to \infty$. Thus,
\[
\lim_{n\to\infty} \E[ (X_n)_a Y_n^b - h_M(X_n, Y_n) ]
\weq (\pi)_{a,b} - \E[h_M(X,Y)].
\]
Now $h_M(X,Y) \le (X)_a Y^b$, so by dominated convergence
\[
\lim_{M\to \infty} \E[ h_M(X, Y) ] \weq (\pi)_{a,b},
\]
and it follows that
\[
\lim_{n\to \infty} (X_n)_a Y_n^b \I (X_n > M) \to 0 \quad \text{as} \quad  M \to \infty.
\]
\end{proof}

\begin{lemma}
\label{lemma_2_communities}
Let $\{E_1, E_2 \}$ be an edge partition of $K_r$, $r\geq 3$. Then either $\norm{E_1} = r$ or $\norm{E_2} = r$. 
\end{lemma}

\begin{proof}
Assume that $\norm{E_i} <r$ for $i= 1,2$.
Let $v_1 \in V(K_r) \setminus V(E_1)$, $v_2 \in V(K_r) \setminus V(E_2)$, and $v \in V(K_r) \setminus \{v_1, v_2\}$. If $v_1 = v_2$, then the edge $\{v_1, v\}$ is not in $E_1$ or $E_2$, and if $v_1 \neq v_2$, then the edge $\{v_1, v_2 \}$ is not in $E_1$ or $E_2$. This is a contradiction, because $\{E_1, E_2 \}$ is an edge partition.
\end{proof}

\begin{lemma}
\label{lemma:uclique}
Let $R = K_r$ be a clique with $r$ nodes.
Assume that $m_n \asymp n$ and \eqref{eq:maxcond} holds.
Then
\[
U(R) \weq (1+O(n^{-1})) m_n \frac{(\pi_n)_{r,\binom{r}{2}}}{(n)_{r}}.
\]
\end{lemma}

\begin{proof}
Denote $q_n = (\pi_n)_{r,s}^{1/s}$, where $s = \binom{r}{2}$ is the edge count of $K_r$.
We observe that $U(R) = \sum_{t=1}^s U_t(R)$ where
\[
 U_t(R)
 \weq \sum_{\phi: |\tran(\phi)| = t} \prod_{k \in \ran(\phi)} \frac{(\pi_n)_{||\phi^{-1}(k)||, \abs{\phi^{-1}(k)}}}{(n)_{||\phi^{-1}(k)||}},
\]
where the double sum is over all the mappings from $E(K_r)$ to $[m_n]$. If $s=1$, the result follows directly from the above identity. Assume that $s\geq 3$.
When $t=1$, we obtain
\begin{align}
\label{eq:cliqueonecover}
 U_1(R)
 \weq \sum_{\phi:|\tran (\phi)|=1} \prod_{k \in \ran(\phi)}
 \frac{(\pi_n)_{||\phi^{-1}(k)||, \abs{\phi^{-1}(k)}}}{(n)_{||\phi^{-1}(k)||}}
 \weq m_n \frac{(\pi_n)_{r, s}}{(n)_{r}}.
\end{align}

Let $t \geq 3$ and $\phi$ be such that $|\tran(\phi)| = t$. 
Now $\cE = \cup_{k \in \tran(\phi) } \{ \phi^{-1}(k) \}$ is a partition of $E(K_r)$. 
Since $K_r$ is connected, we can label the sets $\cE = \{E_1, \ldots, E_{t} \}$ so that $V(E_t) \cap V(E_{t-1}) \neq \emptyset$.
Define $E_{t-1}' = E_{t-1} \cup E_{t}$.
Then
\begin{align*}
\sum_{i=1}^t \norm{E_i} \weq \sum_{i=1}^{t-2} \norm{E_i} + \norm{E_{t-1}} + \norm{E_t} 
\wge \sum_{i=1}^{t-2} \norm{E_i} + \norm{E'_{t-1}} + 1.
\end{align*}
Since $\{E_1, \ldots, E'_{t-1}\}$ is now a partition of $E(K_r)$, we may continue recursively by defining $E'_{i-1} = E_{i-1} \cup E'_{i}$, where $V(E_{i-1}) \cap V(E'_i) \neq \emptyset$, and obtain
\begin{align*}
\sum_{i=1}^t \norm{E_i} \wge \norm{E_1} + \norm{E'_2} + t-2 \wge r +t,
\end{align*}
where the latter inequality follows from Lemma \ref{lemma_2_communities}. If $t=2$,   the same lemma gives $\sum_{i=1}^t \norm{E_i} \wge r +t$ directly.
It follows that for any $t\geq 2$,
\begin{align}
\label{eq:nbound}
\prod_{i=1}^t (n)_{\norm{E_i}} \wge \prod_{i=1}^t \frac{n^{\norm{E_i}}}{\norm{E_i}!} \wge \prod_{i=1}^t \frac{n^{\norm{E_i}}}{r!} \wge (r!)^{-t} n^{r+t}.
\end{align}
By assumption there is a constant $c_r$ such that $(\pi_n)_{\norm{E_i},|E_i|} \le  c_r q_n^{|E_i|}$ for all $E_i \subset E(K_r)$ and all $n$. This together with \eqref{eq:nbound} gives
\begin{align*}
\prod_{i=1}^t \frac{(\pi_n)_{\norm{E_i}, |E_{i}|}}{(n)_{\norm{E_i}}} \wle (r!)^t n^{-(r+t)} \prod_{i =1}^{t} (\pi_n)_{\norm{E_i}, |E_{i}|} &\wle (r!)^t n^{-(r+t)} \prod_{i =1}^{t} c_r q_n^{|E_i|}, \\
&\weq (r!)^t n^{-(r+t)} c_r^t q_n^{s}, 
\end{align*}
where we used  $\sum_{i=1}^t |E_i| = s$. It follows that
\begin{align*}
 \sum_{t=2}^s U_t(R)
 &\wle \sum_{t=2}^s \binom{m_n}{t}t^s (r!)^t n^{-(r+t)} c_r^t q_n^{s}.
\end{align*}
Using $\binom{m_n}{t} \le m_n^t/t!$ and $t \le s$ we further bound
\begin{align*}
 \sum_{t=2}^s U_t(R)
 &\wle (sq_n)^s n^{-r} \sum_{t=2}^s \frac{(r!)^t m_n^t n^{-t}c_r^t}{t!} \\
&\wle  (sq_n)^s n^{-r} (e^{r!  c_rm_n/n} -1 - r!c_rm_n/n).
\end{align*}
We find that $\sum_{t=2}^s U_t(R) \lesim n^{-r} q_n^s \asymp n^{-r} (\pi_n)_{r,s}$,
whereas \eqref{eq:cliqueonecover} shows that $U_1(R) \asymp m_n n^{-r} (\pi_n)_{r,s}$.
This establishes the claim.
\end{proof}

\begin{lemma}
\label{lemma:ucycle}
Let $R = C_r$ be a cycle with $r$ nodes.
Assume that $m_n \asymp n$ and \eqref{eq:maxcond} holds.
Then
\[
 U(C_r)
 \weq (1+O(n^{-1})) m_n \frac{(\pi_n)_{r,r}}{(n)_r}.
\]
\end{lemma}
\begin{proof}
Recall that
\[
U(C_r) \weq \sum_{t=1}^r \sum_{\phi: |\tran(\phi)| = t} \prod_{k \in \ran(\phi)} \frac{(\pi_n)_{||\phi^{-1}(k)||, \abs{\phi^{-1}(k)}}}{(n)_{||\phi^{-1}(k)||}}.
\]
We show that the leading term is obtained with $t=1$:
\begin{align} 
\label{eq:cycleonecover}
 \sum_{\phi: |\tran(\phi)| = 1}
 \prod_{k \in \ran(\phi)} \frac{(\pi_n)_{||\phi^{-1}(k)||, \abs{\phi^{-1}(k)}}}{(n)_{||\phi^{-1}(k)||}}
 = m_n \frac{(\pi_n)_{r,r}}{(n)_r}.
\end{align}

Let $t\geq2$ and $|\tran(\phi)|=t$. Define $\cE = \cup_{k \in \tran(\phi) } \{ \phi^{-1}(k) \} $, where we choose a labeling $\cE = \{E_1, \ldots, E_t \}$ with $V(E_t) \cap V(E_{t-1}) \neq \emptyset$. 

Let $\{e_1, \ldots, e_L \} \subset E_t$ be the longest path in $E_t$. If $L \geq 2$, we choose an ordering such that $e_i \cap e_{i-1} \neq \emptyset$ for $i=2,\ldots, L$.  
Then $e_1 \setminus e_2$ and $e_L \setminus e_{L-1}$ are contained in both $V(E_t)$ and $\cup_{i=1}^{t-1} V(E_i)$. Denoting $V_i = V(E_i)$ and $W_{i-1} = \cup_{j=1}^{i-1} V(E_j)$, it follows that
\begin{align*}
r = | \bigcup_{i=1}^t V_i | &=  \norm{E_1} + | V_t \setminus W_{t-1} | + \sum_{i=2}^{t-1} | V_i \setminus W_{i-1} | \\
& \le  \norm{E_1} + (\norm{E_t} -2) + \sum_{i=2}^{t-1} (\norm{E_i}-1) \\
&= -t+\sum_{i=1}^t \norm{E_i}.
\end{align*}
This bound implies $\prod_{i=1}^t (n)_{\norm{E_i}} \geq (r!)^{-t} n^{r+t}$ by \eqref{eq:nbound}, and so
\[
 \prod_{k \in \ran(\phi)} \frac{(\pi_n)_{||\phi^{-1}(k)||, \abs{\phi^{-1}(k)}}}{(n)_{||\phi^{-1}(k)||}}  \wle (r!)^{t} n^{-(r+t)}  \prod_{k \in \ran(\phi)} (\pi_n)_{||\phi^{-1}(k)||, \abs{\phi^{-1}(k)}}.
\] 
By assumption there is a $c_r$ such that $(\pi_n)_{||\phi^{-1}(k)||, \abs{\phi^{-1}(k)}} \le c_r q_n^{\abs{\phi^{-1}(k)}}$, hence
\[
\prod_{k \in \ran(\phi)} \frac{(\pi_n)_{||\phi^{-1}(k)||, \abs{\phi^{-1}(k)}}}{(n)_{||\phi^{-1}(k)||}}  \wle (r!)^{t} n^{-(r+t)}  c_r^t q_n^r.
\] 
This gives
\begin{align*}
\sum_{t=2}^r \sum_{\phi: |\tran(\phi)| = t} \prod_{k \in \ran(\phi)} \frac{(\pi_n)_{||\phi^{-1}(k)||, \abs{\phi^{-1}(k)}}}{(n)_{||\phi^{-1}(k)||}}  \le \sum_{t=2}^r \binom{m_n}{t} t^r (r!)^{t} n^{-(r+t)}  c_r^t q_n^r,
\end{align*}
where we bound
\begin{align*}
\sum_{t=2}^r \binom{m_n}{t} t^r (r!)^{t} n^{-(r+t)}  c_r^t q_n^r &\weq (q_nr)^r n^{-r} \sum_{t=2}^r m_n^t (t!)^{-1} (r!)^{t} n^{-t}  c_r^t \\
&\wle (q_nr)^r n^{-r} (e^{r! c_r m_n/n}-1-r!c_rm_n/n).
\end{align*}
The claim follows from this together with \eqref{eq:cycleonecover}, and  $c_r \lesim 1$.
\end{proof}

\begin{proof}[Proof of Theorem~\ref{the:piconverges}]
Denote $s = \binom{r}{2}$ and $q_n = (\pi_n)_{r,s}^{1/s}$.
We first verify that $f(K_r) = (1 + O(n^{-1}) ) U(K_r)$. We note that 
\[
 U(K_r)
 \weq \sum_\phi \pr(A_\phi) \wge \sum_{k=1}^m \pr(G_{n,k} \supset R)
 \weq m_n \frac{(\pi_n)_{r,s}}{(n)_r}
 \wasymp m_n n^{-r} (\pi_n)_{r,s}.
\]
On the other hand, by applying Lemma~\ref{the:UpperBoundNew} with $x=1$ and $y=q_n$, we find that
$L(K_r) \lesim n^{-r}q_n^{s+1} \asymp m_n^{-1} q_n U(K_r) \asymp m_n^{-1} U(K_r)$,
and so $f(K_r) = (1+O(n^{-1})) U(K_r)$ follows by Bonferroni's inequalities. 

By Lemma \ref{lemma:uclique} we have $U(K_r) = (1 + O(n^{-1}) ) m_n (\pi_n)_{r, s} (n)_r^{-1}$, so it follows that
\begin{align}
 \label{eq:expectednkr}
 \E N_{K_r}(G_n)
 \weq \frac{(n)_r}{\abs{\Aut(K_r)}} f(K_r)
 &\weq \frac{(n)_r}{r!} (1 + O(n^{-1}q_n))U(K_r) \nonumber \\ 
 &\weq \frac{1}{r!}(1 + O(n^{-1}) ) m_n (\pi_n)_{r, s}. 
\end{align}
The same arguments give the expected value for $C_r$ together with Lemma~\ref{lemma:ucycle}.

We now verify the conditions of Theorem \ref{thm:mainvariance}.
By Assumption \ref{ass:piconverges} and Lemma \ref{lemma:suppim},
the conditions of Lemma \ref{lemma:r2moments} are satisfied, and it follows that
$\sup_n \max_{E \subset E(R)} (\pi_n^M)_{\norm{E}, \abs{E}} \to 0$.
Hence,  Condition~\ref{eq:UI}
of Theorem \ref{thm:mainvariance} is satisfied.
Let now $(X_n,Y_n) \sim \pi_n$ and fix $E \subset E(K_r)$.
By the mapping theorem the random variables $(X_n)_{\norm{E}} Y_n^{|E|}$ converge weakly, and by Skorohod's representation theorem we may assume pointwise convergence. By \eqref{eq:dominate1}--\eqref{eq:dominate3} and Assumption \ref{ass:piconverges}, $(X_n)_{\norm{E}} Y_n^{|E|}$ is dominated by a random variable $Z_n$ (e.g., the right-hand side of \eqref{eq:dominate1}) such that $Z_n \to Z$ pointwise and $\E(Z_n) \to \E(Z) < \infty$. Thus, $(\pi_n)_{\norm{E}, |E|} \to (\pi)_{\norm{E}, |E|}$ by dominated convergence (see e.g.\ \cite{Kallenberg_2002}). From this we also conclude that $(\pi_n)_{\norm{E}, |E|} \lesim 1$ for all $E \subset E(K_r)$.

With $(\pi_n)_{\norm{E}, |E|} = O(1)$ established, the remaining conditions  are straightforward to verify for $K_r$. By assumption, $(\pi)_{r,\binom{r}{2}}>0$ holds, so it follows that $q_n = \Theta(1)$ (satisfying the condition $q_n \gg n^{-1/s}$) and
\[
\max_{E \subset E(K_r)} \frac{(\pi_n)_{\norm{E},|E|}}{q_n^{|E|}} \weq O(1),
\]
and
\[
 \max_{0 \le b \le 2s} \frac{(\pi_n)_{0,b}}{q_n^b} \le  \max_{0 \le b \le 2s} q_n^{-b} \weq O(1).
\]
The argument for $C_r$ is identical. Thus, the conditions of Theorem~\ref{thm:mainvariance} are satisfied, yielding $N_R(G_n) = (1+o_\pr(1)) \E N_R (G_n)$ for $R=K_r, C_r$.
\end{proof}

\begin{proof}[Proof of Theorem~\ref{the:sizeconverges}]
We prove the claim for $K_r$, as the proof remains identical for $C_r$.
The argument for $f(K_r) = (1+O(n^{-1}q))U(K_r)$ is the same as in the proof of Theorem \ref{the:piconverges}.
For the expression of the expected value we refer to \eqref{eq:expectednkr} and note that
\begin{equation}
\label{eq:pmoment}
(\pi_n)_{r,\binom{r}{2}} = \E[(X_n)_r Y_n^{\binom{r}{2}}] = \E[(X_n)_r] p_n^{\binom{r}{2}}.
\end{equation}
We now verify the conditions of Theorem \ref{thm:mainvariance}. The condition $q_n \gg n^{-1/s}$ is satisfied for $s=\binom{r}{2}$, directly by the assumptions $n p_n^{\binom{r}{2}} \gg 1$ and $\E[(X_n)_r] = \Theta(1)$, and \eqref{eq:pmoment}. 
As for the other conditions, we have
\[
 \frac{(\pi_n)_{\norm{E},|E|}}{q_n^{|E|}}  \weq \Theta(1) \frac{\E[(X_n)_{\norm{E}}] p_n^{|E|} }{\E[(X_n)_r]^{|E|/s} p_n^{|E|}} \weq O(1), 
\]
where we used the fact that $\norm{E} \le r$, and
\[
(\pi_n)_{0,2s} \weq p_n^{2s} \wlesim p_n^{2s}\E[(X_n)_r]^2 \weq (\pi_n)_{r,s}^2.
\]
Finally, the random variables $(X_n)_{\norm{E}}$ are uniformly integrable because $(X_n)_r$ are uniformly integrable and
$(X_n)_{\norm{E}} \le r! + (X_n)_r$. Thus,
\[
\sup_n (\pi_n^M)_{\norm{E}, \abs{E}} \wle \sup_n \E[(X_n)_{\norm{E}} \I(X_n>M)] \to 0
\]
as $M \to \infty$. The claim now follows by Theorem \ref{thm:mainvariance}.
\end{proof}

\section*{Acknowledgements}
This work was supported by COSTNET COST Action CA15109. JK was supported by the Magnus Ehrnrooth Foundation and Academy of Finland grant 346311 -- Finnish Centre of Excellence in Randomness and Structures.

\bibliographystyle{abbrvnat}
\bibliography{lslReferences}

\end{document}